%% file: main.tex
\documentclass[11pt]{amsart}
\usepackage{silence}
\usepackage{mathtree}
\author{Chiara Ascenzi}
\address{Tampere University, Tampere, Finland}
\email{chiara.ascenzi@tuni.fi}
\title{Derived equivalences between diagram categories of finite posets}

\begin{document}

\restorepagecolor




\begin{abstract}
    We study universal derived equivalences between diagram categories indexed by finite posets. Starting from a construction of Ladkani, we give an intrinsic criterion for determining when a finite poset admits a decomposition to which this construction can be applied. This leads to the notion of an admissible cut, formulated entirely in terms of the order structure of the poset. Our main result therefore provides a method for producing, from a given finite poset admitting such a cut, a new poset that is universally derived equivalent to it. The construction is reversible once the partition is retained, so that the original mixed order relations can be recovered from the transformed poset. As applications, we show that every finite poset of height at most one is universally derived equivalent to its opposite, give a criterion characterizing source-to-sink transformations at minimal elements, and recover the universal derived equivalence of all orientations of a finite tree through sequences of such local transformations. These results are also applied to persistence modules indexed by finite posets and to extensions obtained by adjoining further finite parameters.
    \end{abstract}
    \maketitle
\input{introduzione}

\input{Raccolta_teoremi_UDE}
\input{Teorema_principale_e_conseguenze}
\input{orientation_of_tree}
\input{TDA}

\input{bibliography}
\end{document}

%% file: introduzione.tex
\section{Introduction}
Derived categories play a central role in homological algebra and representation theory, where they provide a framework for comparing algebraic objects beyond their abelian or module categories. In the representation theory of quivers and posets, one is naturally led to derived categories of diagram categories. If $P$ is a finite poset and $\mathcal{A}$ is an abelian category, the category $\mathcal{A}^P$ of $P$-shaped diagrams in $\mathcal{A}$ is again abelian and one may study its derived category $D(\mathcal{A}^P)$. 

For a fixed coefficient category -- for instance, finite-dimensional vector spaces over a field $k$ -- one may ask when two posets give rise to equivalent derived categories of representations. Such an equivalence is, a priori, tied to the chosen coefficients. A stronger phenomenon occurs when the equivalence is completely independent of this choice: two finite posets $P$ and $Q$ are called universally derived equivalent if
\[D(\mathcal{A}^P) \simeq D(\mathcal{A}^Q)\] for every abelian category $\mathcal{A}$. 
Universal derived equivalence is therefore a genuinely combinatorial property of the indexing posets. The problem is then to understand which transformations of a finite poset preserve its derived category in this universal sense. 

One of the main constructions of universally derived equivalent posets is due to Ladkani \cite{ladkani}. Starting from two finite posets together with suitable combinatorial data relating them, his construction produces two different partial orders on their disjoint union and proves that the corresponding diagram categories are universally derived equivalent (see Section \ref{Ladkani's construction}).

In this article, we address the corresponding recognition problem. Starting with a finite poset $P$, we determine exactly which partitions $P=A \sqcup B$ realize $P$ as the positive poset in Ladkani's construction. Our main result, Theorem \ref{generalizzazione teomio}, gives necessary and sufficient conditions for such a realization, formulated entirely in terms of the order structure of $P$. These conditions define an admissible cut (Definition \ref{def admissible cuts}): $A$ is an ideal, $B$ is a filter, and the sets $M(a)$ of minimal elements of $B$ lying above $a$ satisfy separation and compatibility conditions.

Whenever such a realization exists, the subsets required by Ladkani's construction are uniquely determined by $P$ and by the chosen partition. Moreover, such subsets coincide with the sets $M(a)$. Thus, for each partition, the construction data can be recovered from the given order and their admissibility checked directly. Testing all partitions gives a finite procedure for finding all admissible cuts of $P$ (see the discussion following Theorem \ref{generalizzazione teomio}). Each admissible cut then determines an associated poset $P^-$, obtained by changing the relations between $A$ and $B$, called mixed relations. This associated poset is universally derived equivalent to $P$ (Corollary \ref{The associated universally derived equivalent poset}). For example, applying this procedure to the five-element Tamari lattice $T_3$ produces a universally derived equivalent poset whose Hasse diagram is an orientation of the Dynkin diagram $D_5$ (Example \ref{T3/D5}). In particular, the underlying undirected Hasse graph changes from a cycle to a tree.

The transformation is reversible once the partition $A \sqcup B$ is retained: the sets $M(a)$, and hence all mixed relations of $P$, can be recovered from $P^-$ (Proposition \ref{Reversibility of the construction}). This reconstruction is dual to the original construction. We also show that every finite poset of height at most one is universally derived equivalent to its opposite (Corollary \ref{Two-level posets and opposites}).

In the singleton case, where each $M(a)$ consists of a single element, the construction is determined by an order-preserving map $A \to B$ (Corollary \ref{Reconstructing a poset from a cut}). The non-singleton case is essential, however, and appears naturally in the source-to-sink operations studied in Section \ref{An intrinsic criterion for source-clicks}. We also establish an invariance property under isomorphisms (Proposition \ref{invgen}): transporting the combinatorial data along isomorphisms produces isomorphic posets, reducing redundancy in classification problems.

We next consider local transformations of Hasse diagrams. For a minimal element $s$, we reverse all arrows incident with $s$ in the Hasse diagram and take the transitive closure, obtaining a new poset in which $s$ is maximal. This operation is called a source-click at $s$ (see Section \ref{An intrinsic criterion for source-clicks}). We give a criterion, valid for arbitrary finite posets, for this source-click to arise from an admissible cut (Proposition \ref{General source-click criterion}). For orientations of trees the criterion is automatically satisfied (Corollary \ref{Source step}). We then give a short combinatorial proof of the known fact that any orientation of a finite tree can be reached from any other by a finite sequence of source-clicks (Proposition \ref{Transitivity of clicks on a tree}). In this way, the general construction recovers the universal derived equivalence of all orientations of a tree, while also explaining it through explicit local transformations (Theorem \ref{tree-exhaustiveness}).

Finally, we relate these constructions to persistence modules, which are a central objects in topological data analysis (TDA). Since persistence modules indexed by a finite poset are diagrams of vector spaces over that poset, universal derived equivalence induces derived equivalences between the corresponding categories of pointwise finite-dimensional persistence modules (see Section \ref{Application to persistence modules and parameter extension}). We also use the stability of universal derived equivalence under products (Proposition \ref{Stability under parameter extension}) to adjoin further finite parameters to the pairs produced by admissible cuts (Corollary \ref{Parameter extension for admissible cuts}). This yields infinite families of derived equivalent parameter posets (Example \ref{example Tamari TDA}) and extends the tree-orientation results to multiparameter settings (Corollary \ref{Parameter extension for tree orientation}).

\subsection*{Organization.}
Section 2 introduces the notation and terminology used throughout the paper and recalls Ladkani's construction. Section 3 proves the main characterization theorem, shows that the associated construction is reversible, derives an application to posets of height at most one and their opposites, discusses the singleton specialization and cardinality constraints, and records an invariance property under isomorphisms. Section 4 gives the general criterion for source-clicks to be realized by admissible cuts and then applies this criterion to orientations of finite trees. Section 5 relates these results to persistence modules and studies their stability under parameter extension, with applications to the Tamari example and to orientations of finite trees.

%% file: Raccolta_teoremi_UDE.tex
\section{Preliminaries}

In this section, we introduce the notation and terminology used throughout the paper. We first recall diagram categories indexed by finite posets and the notion of universal derived equivalence. We then recall Ladkani's construction, which provides the main background for the constructions developed in the following section.

For general background on derived categories and homological algebra, see for instance  Weibel \cite{weibel} and Miličić \cite{dercat}.
For module categories, derived equivalences are classically studied through tilting theory and Rickard's derived Morita theory \cite{Rickard}.

\subsection{Diagram categories and universal derived equivalence}

Let $X$ be a finite poset. We regard $X$ as a category whose objects are the elements of $X$, and in which there is a unique morphism $x \to x'$ whenever $x \leq x'$.

The height $h(X)$ of a finite poset $X$ is the largest integer $r$ such that $X$ contains a chain $x_0<x_1< \dots <x_r$ of $r+1$ distinct elements.

If $\mathcal{A}$ is an abelian category, we denote by $\mathcal{A}^X$ the category of $X$-shaped diagrams in $\mathcal{A}$, that is, the functor category $\operatorname{Fun}(X,\mathcal{A})$. Thus, an object of $\mathcal{A}^X$ consists of a family $(A_x)_{x \in X}$ of objects of $\mathcal{A}$, together with morphisms $A_x \to A_{x'}$ for every relation $x \le x'$, compatible  with composition. Since kernels and cokernels in $\mathcal{A}^X$ are computed pointwise, $\mathcal{A}^X$ is abelian whenever $\mathcal{A}$ is abelian. We write $D(\mathcal{A}^X)$ for its derived category. 

Two finite posets $X$ and $Y$ are said to be universally derived equivalent \cite{ladkani} if, for every abelian category $\mathcal{A}$, there is an equivalence of triangulated categories
\[D(\mathcal{A}^X) \simeq D(\mathcal{A}^Y).\] 

Universal derived equivalence is therefore a combinatorial property of the indexing posets, independent of the choice of coefficients or of a particular category of representations.
\subsection{Ladkani's construction} \label{Ladkani's construction}

Given two posets $X$ and $Y$, for $y \in Y$ we define \[[y, \bullet ] \coloneq \{ y' \in Y \,\,\text{such that}\,\,y \leq y' \},\]
\[[\bullet, y ] \coloneq \{ y' \in Y \,\,\text{such that}\,\,y'\leq y \}.\]

Let $X$ and $Y$ be two finite posets and $\{Y_x\}_{x \in X}$ a collection of subsets of $Y$ such that the following properties hold.

\begin{enumerate}[label=(\roman*), ref=(\roman*)]
\item \label{prop:disjoint} $\forall x \in X, \,\forall y \neq y' \in Y_x $, \[
     [y, \bullet ] \cap [y', \bullet ]= \emptyset \text{\,\,\,\,and\,\,\,\,} [\bullet, y ] \cap [\bullet, y' ]= \emptyset;
\]

    \item \label{prop:phi} $\forall x,x' \in X$ such that $x \leq x'$, we have a bijection
    $\varphi_{xx'} \colon Y_x \to Y_{x'}$ such that $\forall y \in Y_{x}$, 
    $
    y \leq \varphi_{xx'}(y)$.
\end{enumerate}

Now, we define two relations $\leq_+$ and $\leq_-$ on the disjoint union $X \sqcup Y$.
These relations coincide with the original orders on $X$ and on $Y$ (i.e. if $x\leq x'$ in $[X,\leq]$, then $x \leq_+x'$ and $x \leq_- x'$, and similarly for elements inside $[Y,\leq]$).
Moreover, if $x\in X$ and $y \in Y$, we define:
\[x \leq_+y \iff \exists \,  y_x \in Y_x \text{\,\,such that\,\,}y_x \leq y,\]
\[y \leq_-x \iff \exists \,  y_x \in Y_x \text{\,\,such that\,\,}y \leq y_x.\]
These are the only mixed relations; in particular, for $x \in X$ and $y \in Y$, we never have $y \leq_+x$ or $x \leq_-y$.

Let $\mathcal{Y}=(Y_x)_{x \in X}$. We denote $X \sqcup Y$, equipped with the relations $\leq_+$ and $\leq_-$, by 
\[(X \sqcup Y)^+_{\mathcal{Y}}\,\,\,\,\, \text{and}\,\,\,\,\,(X \sqcup Y)^-_{\mathcal{Y}},\] respectively. When the family $\mathcal{Y}$ is clear from the context, we omit the subscript.

\begin{lemma}\label{lemma part ord}\cite[Section 1]{ladkani}
    The relations $\leq_+$ and $\leq_-$ defined above are partial orders on $X \sqcup Y$. Hence  $(X \sqcup Y)^+_{\mathcal{Y}}$ and $(X \sqcup Y)^-_{\mathcal{Y}}$ are posets. 
\end{lemma}
We call $(X \sqcup Y)_{\mathcal{Y}}^+$ and $(X \sqcup Y)^-_{\mathcal{Y}}$ the positive and negative posets associated with $\mathcal{Y}$, respectively.

\begin{theorem}\label{ladkani}

\cite[Theorem~1.1]{ladkani}
The positive and negative posets $(X \sqcup Y)^+_{\mathcal{Y}}$ and $(X \sqcup Y)^-_{\mathcal{Y}}$ are universally derived equivalent.

\end{theorem}

If each $Y_x$ is a singleton, say $Y_x=\{f(x)\}$, then the family $\{Y_x\}_{x \in X}$ is uniquely determined by a function $f \colon X \to Y$. In this case, assumption \cref{prop:disjoint} is automatic, while assumption \cref{prop:phi} is equivalent to $f$ being order-preserving. Thus, Theorem \ref{ladkani} specializes to the following result.

\begin{proposition}\cite[Corollary~1.3]{ladkani}\label{Ladkani singleton}
    Let $X$ and $Y$ be finite posets, and let $f \colon X \to Y$ be an order-preserving map. Define two partial orders $\leq_f^+$ and $\leq_f^-$ on the disjoint union $X \sqcup Y$ by requiring that they restrict to the given orders on $X$ and $Y$, and that, for $x \in X$ and $y \in Y$, 
    \[x \leq_f^+ y \iff f(x) \leq y, \,\,\,\,\,\,\,\, y \leq_f^- x \iff y \leq f(x).\]
    Then the two posets $(X \sqcup Y, \leq_f^+)$ and $(X \sqcup Y, \leq_f^-)$ are universally derived equivalent.
\end{proposition}

Ladkani later referred to this order-preserving map construction as a flip-flop \cite{ladkani2}; see also \cite[Section 5]{Goguet} for a recent use of this terminology.

%% file: teorema_principale_e_conseguenze.tex
\section{Admissible cuts and derived equivalences}

This section contains the main result of the paper. We show how, starting from a finite poset $P$, one can recognize directly from its order structure decompositions to which Theorem \ref{ladkani} can be applied, and hence construct new posets universally derived equivalent to $P$. More precisely, given a partition $P=A \sqcup B$, we characterize when the data required by that theorem are already encoded in the order structure of $P$. This leads to the intrinsic notion of an admissible cut.

Once such an admissible cut is found, the required data are recovered from $P$ itself and determine an associated poset $P^-$ universally derived equivalent to $P$. We show that this transformation is reversible once the partition $A \sqcup B$ is retained, and we illustrate the construction with the Tamari lattice $T_3$. We then derive an application to posets of height at most one, study the singleton specialization and some constraints on the cardinalities of the sets $M(a)$ and finally record an invariance property under isomorphisms. 

Throughout this section, $P$ denotes a finite poset and $P=A \sqcup B$ a partition of its underlying set. Whenever needed, $A$ and $B$ are endowed with the partial orders induced from $P$, denoted by $\leq_A$ and $\leq_B$, respectively.

\subsection{The general construction}

Recall that a subset $A \subseteq P$ is called an \textbf{order ideal} if, whenever $a \in A$ and $ x \leq_P a $, then $ x \in A$. Dually, a subset $B \subseteq P$ is called an \textbf{order filter} if, whenever $b \in B $ and $ b \leq_P x$, then $x \in B$. From now on, we simply write ideals and filters.

With the above notation, for a subset $S \subseteq P$, define
\[U_B(S) \coloneq \{b \in B \colon \text{ there exists } s \in S \text{ such that } s \leq_Pb\}.\] When $S=\{x\}$ is a singleton, we write $U_B(x)$ instead of $U_B(\{x\})$. If $S \subseteq B$, then $U_B(S)$ is the upward closure of $S$ in $B$.

For a subset $C \subseteq B$, let $\operatorname{Min}_B(C)$ and $\operatorname{Max}_B(C)$ denote the sets of minimal and maximal elements of $C$, respectively, with respect to $\leq_B$. For every $a \in A$, set
\[M(a) \coloneq \operatorname{Min}_B(U_B(a)).\]

For $ b\in B$, we write 
\[[b, \bullet]_B \coloneq \{b' \in B \colon b \leq_Bb'\}, \,\,\,\,\, [\bullet,b]_B \coloneq \{b'\in B \colon b' \leq_Bb\}.\]

\begin{definition} 
\label{def admissible cuts}
   Let $P$ be a finite poset and let $P=A \sqcup B$ be a partition of its underlying set. 
   We say that $P=A \sqcup B$ is an \textbf{admissible cut} if the following conditions are satisfied:
\begin{enumerate}
    \item [(a)] Ideal-filter condition: $A$ is an ideal of $P$ and $B$ is a filter of $P$;
    \item [(b)]  Separation condition: for every $ a \in A$ and every pair of distinct elements $m, m' \in M(a)$, \[[m, \bullet]_B \cap [m', \bullet]_B = \emptyset \,\,\,\, \text{ and } \,\,\,\, [\bullet, m]_B \cap [\bullet, m']_B = \emptyset.\]
    \item [(c)] Compatibility condition: for every $a \leq_A a'$ and every $m \in M(a)$, there exists $m'\in~M(a')$ such that $m \leq_B m'$.
\end{enumerate}
\end{definition}

\begin{theorem} 
\label{generalizzazione teomio}
    Let $P$ be a finite poset, and let $P=A \sqcup B$ be a partition of its underlying set.

     The following are equivalent: 
     \begin{enumerate}
    \item There exists a family $\mathcal{B}=(B_a)_{a \in A}$ of subsets of $B$ satisfying the hypotheses of Ladkani's construction such that $P=(A \sqcup B)^+_{\mathcal{B}}$.
    
    \item $P=A \sqcup B$ is an admissible cut.
    \end{enumerate}
   Moreover, whenever these equivalent conditions hold, the subsets $B_a$ in $(i)$ are necessarily given by 
   \[B_a=M(a) \,\, \text{ for every } \,\, a \in A.\]

\end{theorem}

\begin{proof}
Let $P$ be a finite poset.
\item[\boxed{(i)\Rightarrow(ii)}:]
Assume that a family $(B_a)_{a \in A}$ as in $(i)$ exists. By assumption, for every $a \in A$ and $b \in B$, \begin{equation} \label{defP+}
 a \leq_P b \iff \text{ there exists }\, y \in B_a \text{ such that } y \leq_B b.
 \end{equation}
 \begin{itemize}
 \item \textbf{\underline{Step 1: Ideal-filter condition.}} 
 
 In the poset $(A \sqcup B, \leq_+)$ the only mixed  relations are of the form \[a \leq_+b, \,\,\,\,\,\,\,\, a \in A,\,\, b \in B.\]
 
  If $x \leq_P a$ with $a \in A$, then $x$ cannot belong to $B$. Hence $x \in A$, and therefore $A$ is an ideal of $P$. Dually, if $b \leq_Px$ with $b \in B$, then $x$ cannot belong to $A$. Hence $x \in B$, and therefore $B$ is a filter of $P$.
 
\item \textbf{\underline{Step 2: $B_a=M(a)$ for every $a \in A$.}}

Fix $a \in A$. By (\ref{defP+}), we have \[U_B(a)=\{b \in B \colon a \leq_P b\}=\{b \in B \colon \text{ there exists } \, y \in B_a \text{ such that } y \leq_B b\}=U_B(B_a).\]
We first prove that $B_a \subseteq M(a)$. Let $y \in B_a$. Since $y \leq_By$, we have $y \in U_B(B_a)=U_B(a)$. Suppose that $z \in U_B(a)$ and $z \leq_B y$. Since $z \in U_B(a)=U_B(B_a)$, there exists $y' \in B_a$ such that $y' \leq_B z \leq_B y$. If $y' \neq y$, then $y \in [y', \bullet]_B \cap [y, \bullet]_B$, contradicting condition (i) of the construction in Section \ref{Ladkani's construction}, applied to the distinct elements $y,y' \in B_a$. Hence $y'=y$, and from $y \leq_B z \leq_B y$ we get $z=y$. So $y$ is minimal in $U_B(a)$, i.e. $y \in M(a)$.

Conversely, let $z \in M(a)$. Since $z \in U_B(a)= U_B(B_a)$, there exists $y \in B_a$ such that $y \leq_Bz$. By the inclusion already proved, $y \in M(a)$. Since $z$ is minimal in $U_B(a)$, while $y \in U_B(a)$ and $y \leq_B z$, we obtain $y=z$. Hence $z \in B_a$.

Therefore, $B_a=M(a)$.

\item \textbf{\underline{Step 3: Separation condition.}} Since $B_a=M(a)$ for every $a \in A$, condition (b) is exactly condition (i) of the construction in Section \ref{Ladkani's construction} applied to the family $(B_a)_{a \in A}$.

\item \textbf{\underline{Step 4: Compatibility condition.}} Let $a \leq a'$ in $A$. By condition (ii) of the construction in Section \ref{Ladkani's construction}, there exists a bijection $\varphi_{a,a'} \colon B_a \to B_{a'}$ such that $y \leq_B \varphi_{a,a'}(y)$ for every $y \in B_a$. Since $B_a=M(a)$ and $B_{a'}=M(a')$, for every $m \in M(a)$ the element $\varphi_{a,a'}(m)$ belongs to $M(a')$ and satisfies $m \leq_B \varphi_{a,a'}(m)$. Thus condition (c) holds.
\end{itemize}

Therefore, the partition $P=A \sqcup B$ satisfies the conditions $(a)$, $(b)$ and $(c)$, and is an admissible cut.
 
 \item[\boxed{(ii)\Rightarrow(i)}:] Assume that $P=A \sqcup B$ is an admissible cut. Define, for each $a  \in A$, \[B_a \coloneq M(a).\]
 We show that $(B_a)_{a \in A}$ satisfies the hypotheses of Section \ref{Ladkani's construction} and that the resulting positive poset coincides with $P$. 
\begin{itemize}
 \item \textbf{\underline{Step 1: Verification of condition (i).}} By definition $B_a=M(a)$. For every $a \in A$ and every $m \neq m' \in B_a=M(a)$,
 \[[m, \bullet]_B \cap [m', \bullet]_B = \emptyset, \,\,\,\,\, [\bullet,m]_B \cap [\bullet,m']_B = \emptyset.\]
 \item \textbf{\underline{Step 2: Verification of condition (ii).}} Let $a \leq a'$ in $A$. 
 
 We construct $\varphi_{a,a'} \colon M(a) \to M(a')$ with $m \leq_B \varphi_{a,a'}(m)$ for every $m \in M(a)$, and show it is a bijection.

 By condition (c), for each $m \in M(a)$ there exists at least one $m' \in M(a')$ with $m \leq_Bm'$.

 \textit{Uniqueness of $m'$}: suppose $m \leq_B m'_1$ and $m \leq_B m_2'$ with $m_1',m_2' \in M(a')$. Then $m \in [\bullet,m_1']_B \cap [\bullet,m_2']_B$. 
 
 By condition (b) applied to $a'$, this intersection is empty whenever $m_1' \neq m_2'$. Hence $m_1' = m_2'$.

 Then we set $\varphi_{a,a'}(m) \coloneq m'$, where $m'$ is the unique element in $M(a')$ with $m \leq_Bm'$.
 By construction, $m \leq_B \varphi_{a,a'}(m)$ for every $m \in M(a)$.
 
 \textit{Injectivity:} Suppose $\varphi_{a,a'}(m_1)=\varphi_{a,a'}(m_2)=m'$. Then $m' \in [m_1, \bullet]_B \cap [m_2, \bullet]_B$. By condition (b) applied to $a$, this intersection is empty whenever $m_1 \neq m_2$. Hence $m_1=m_2$.

 \textit{Surjectivity:} Let $m' \in M(a')$. From $a \leq_P a' \leq_Pm'$, we have $m' \in U_B(a)$. Since $P$ is finite, every element of a nonempty subset lies above at least one minimal element of that subset. Applied to $U_B(a)$, this gives an element $m \in M(a)$ such that $m \leq_Bm'$. By uniqueness, $\varphi_{a,a'}(m)=m'$. 

 Therefore $\varphi_{a,a'}$ is a bijection satisfying \[m \leq_B \varphi_{a,a'}(m)\] for every $m \in M(a)$. Hence the family $(B_a)_{a \in A}$, with $B_a=M(a)$, satisfies condition $(ii)$ of Section \ref{Ladkani's construction}.

 \item \textbf{\underline{Step 3: Reconstruction of the mixed relations.}}
 Let $P^+=(A \sqcup B, \leq_+)$ be the positive poset associated with the family $B_a=M(a)$. The orders inside $A$ and $B$ coincide in $P$ and $P^+$ by construction, so it suffices to compare the mixed relations.

 Let $a \in A$ and $b \in B$. By definition of $\leq_+$, 
 \begin{equation} \label{defP+M(a)} 
 a \leq_+b \iff \exists m \in M(a) \text{ with } m \leq_Bb.
 \end{equation}
 Since $U_B(a)$ is upward closed in $B$ and $B$ is finite, every element of $U_B(a)$ lies above a minimal element of $U_B(a)$. Since \[M(a) = \operatorname{Min}_B(U_B(a)),\] it follows that \[U_B(a)=U_B(M(a)).\]

 Combining (\ref{defP+M(a)}) with the equality $U_B(a)=U_B(M(a))$, we obtain, for every $a \in A$ and $b \in B$, \[a \leq_+b \iff b \in U_B(M(a)) \iff b \in U_B(a) \iff a \leq_Pb.\]

 Finally, there are no mixed relations of the form $b \leq a$ (with $b \in B, a \in A$) in $P$ (because $A$ is an ideal by (a)) nor in $P^+$ (by construction). Hence $P$ and $P^+$ have the same partial order, i.e., $P^+=P$.

 \end{itemize}
\end{proof}
For a fixed admissible cut $P=A \sqcup B$, let $\mathcal{M}=(M(a))_{a \in A}$. We denote the positive and negative posets associated with this family by \[P^+ \coloneq (A \sqcup B)^+_{\mathcal{M}}, \,\,\,\,\,\,\, P^- \coloneq (A \sqcup B)^-_{\mathcal{M}}.\] By Theorem \ref{generalizzazione teomio}, $P^+=P$.

Since $P$ is finite, Theorem \ref{generalizzazione teomio} gives a finite procedure for detecting admissible cuts. For a given partition $P=A \sqcup B$, one computes the sets $M(a)$ and checks conditions (a)--(c). By running this test over the finitely many partitions of the underlying set of $P$, one can determine all admissible cuts of $P$ and hence all associated negative posets $P^-$.

 The following corollary describes $P^-$ explicitly.

\begin{corollary} 
\label{The associated universally derived equivalent poset}
Let $P$ be a finite poset and let $P=A \sqcup B$ be an admissible cut. Then the mixed relations of the associated negative poset $P^-$ are given by
\[b \leq_-a \iff \text{ there exists } m \in M(a) \text{ such that } b \leq_B m.\]
Moreover, $P^-$ is universally derived equivalent to $P$.
\end{corollary}

\begin{proof}
The description of the mixed relations follows directly from the definition of $P^-$.
By Theorem \ref{generalizzazione teomio}, $P^+=P$. By Lemma \ref{lemma part ord}, $P^-$ is a poset, and by Theorem \ref{ladkani}, $P^+$ and $P^-$ are universally derived equivalent. Hence $P$ and $P^-$ are universally derived equivalent.
\end{proof}

The next result shows that the construction is reversible once the partition $A \sqcup B$ is retained. More precisely, the data $M(a)$, and hence the original mixed relations of $P$, can be recovered from the associated negative poset $P^-$.

\begin{proposition}
\label{Reversibility of the construction}

Let $P$ be a finite poset, let $P=A \sqcup B$ be an admissible cut, and let $P^-$ be the associated negative poset. Then, in $P^-$, $B$ is an ideal and $A$ is a filter. For every $a \in A$, set \[L_B^-(a) \coloneq \{b \in B \colon b \leq_-a\}.\]
Then
\[M(a)=\operatorname{Max}_B(L_B^-(a)).\]
Consequently, the original poset $P$ is uniquely determined by $P^-$ together with the partition $A \sqcup B$. More precisely, the orders induced on $A$ and $B$ are unchanged, and, for every $a \in A$ and $b \in B$,
\[a \leq_Pb \iff \text{ there exists } m \in \operatorname{Max}_B(L^-_B(a)) \text{ such that } m \leq_B b.\]
    
\end{proposition}

\begin{proof}
By the definition of the negative order, the only mixed relations in $P^-$ are of the form \[b \leq_-a, \,\,\,\, b \in B, \, a \in A.\]
Hence no element of $A$ lies below an element of $B$ in $P^-$. It follows that $B$ is an ideal of $P^-$. Dually, $A$ is a filter of $P^-$.

Now fix $a \in A$. By the definition of the negative order, \[L^-_B(a)= \bigcup_{m \in M(a)}[\bullet,m]_B.\]
    We first prove that every maximal element of $L^-_B(a)$ belongs to $M(a)$. Let $b\in L_B^-(a)$ be maximal. Then there exists $m \in M(a)$ such that $b \leq_Bm$. Since $m \in L_B^-(a)$, the maximality of $b$ implies $b=m$. Therefore
    \[\operatorname{Max}_B(L^-_B(a)) \subseteq M(a).\] Conversely, let $m \in M(a)$. Then $m \in L_B^-(a)$. Suppose that $m \leq_Bb$ for some $b \in L_B^-(a)$. By the definition of $L^-_B(a)$, there exists $m' \in M(a)$ such that $b\leq_B m'$. Hence $m \leq_Bb\leq_Bm'.$ Since $m,m' \in M(a)$, and $M(a)$ consists of the minimal elements of $U_B(a)$, the relation $m \leq_Bm'$ implies $m=m'$. It follows that $b=m$. Hence $m$ is maximal in $L^-_B(a)$, and therefore \[M(a) \subseteq \operatorname{Max}_B(L_B^-(a)).\]
    This proves \[M(a)=\operatorname{Max}_B(L^-_B(a)).\]
    The passage from $P$ to $P^-$ does not change the orders induced on $A$ and $B$. Moreover, since $A$ is an ideal of $P$, there are no mixed relations of the form $b \leq_Pa$, with $b \in B$ and $a \in A$.

    It remains to recover the mixed relations from $A$ to $B$. Since $B$ is finite and $U_B(a)$ is upward closed, it is the upward closure of its minimal elements. Hence
    \[a \leq_Pb \iff b \in U_B(a) \iff \text{ there exists } m \in M(a) \text{ such that } m\leq_Bb.\]
    Using the equality proved above, \[M(a)=\operatorname{Max}_B(L_B^-(a)),\] we obtain  
    \[a \leq_Pb \iff \text{ there exists } m \in \operatorname{Max}_B(L^-_B(a)) \text{ such that } m \leq_B b.\]
    Thus all the relations of $P$ can be reconstructed from $P^-$ and the partition $A \sqcup B$.
\end{proof}

Thus, once the partition $A \sqcup B$ is retained, passing from $P$ to $P^-$ does not lose any information. Notice that the roles of the two parts are reversed: $A$ is an ideal and $B$ is a filter in $P$, whereas $B$ is an ideal and $A$ is a filter in $P^-$. Correspondingly, the elements of $M(a)$, which are the minimal elements of $B$ lying above $a$ in $P$, are recovered as the maximal elements of $B$ lying below $a$ in $P^-$. In this sense, the reconstruction is dual to the original construction.

This duality can also be expressed by passing to opposite posets. Indeed, if $Q=(P^-)^{\operatorname{op}},$ then $A$ is an ideal and $B$ is a filter in $Q$, and for every $a \in A$, 
\[\operatorname{Min}_{B^{\operatorname{op}}}\{b \in B \colon a \leq_Qb\}= \operatorname{Max}_B(L^-_B(a))=M(a).\]
Thus, the data recovered from $P^-$ are precisely the order-dual counterparts of the data used in the original construction of Theorem \ref{generalizzazione teomio}.

\begin{example} \label{T3/D5}
    Consider the five-element Tamari lattice $T_3$ \cite{Geyer}, labelled as in the diagram on the left. Let 
    \[A=\{1,2,3,4\}, \,\,\, B=\{5\}.\]
    Since $5$ is the greatest element of $T_3$, the subset $A$ is an ideal and $B$ is a filter. Moreover, for every $a \in A$,
    \[U_B(a)=\{5\}, \,\,\, \text{ and hence }\,\,\, M(a)=\{5\}.\]
    The separation condition is vacuous, while the compatibility condition is immediate. Therefore $T_3=A \sqcup B$ is an admissible cut.
\[\begin{tikzcd}[sep=small]
	&& {5 \bullet} &&&&&&&&&&&& \\
	{4 \bullet} &&&& {3 \bullet} &&&& {4 \bullet} && {2\bullet} && {1 \bullet} && {3 \bullet} \\
	\\
	{2 \bullet} &&&&&&&&&&&& {5 \bullet} \\
	&& {1 \bullet}
	\arrow[from=2-1, to=1-3]
	\arrow[from=2-5, to=1-3]
	\arrow[from=2-11, to=2-9]
	\arrow[from=2-13, to=2-11]
	\arrow[from=2-13, to=2-15]
	\arrow[from=4-1, to=2-1]
	\arrow[from=4-13, to=2-13]
	\arrow[from=5-3, to=2-5]
	\arrow[from=5-3, to=4-1]
\end{tikzcd}\]
The poset displayed on the right is the associated negative poset $T_3^-$. Indeed, the order induced on $A$ is unchanged, while $5 \leq_-a$ for every $a \in A$. By Corollary  \ref{The associated universally derived equivalent poset}, $T_3$ and $T_3^-$ are universally derived equivalent.

The underlying undirected Hasse graph of $T_3$ is a cycle, whereas that of $T_3^-$ is the Dynkin tree $D_5$. By Proposition \ref{Reversibility of the construction}, however, the original Tamari lattice $T_3$ can be reconstructed from $T_3^-$ together with the partition $A \sqcup B$. Indeed, the unique element $5 \in B$ is the maximal element of $B$ lying below every $a \in A$ in $T_3^-$, so that $M(a)=\{5\}$ is recovered for every $ a \in A$.
\end{example}

\begin{remark}
    In Definition \ref{def admissible cuts}, condition $(c)$ does not follow from conditions $(a)$ and $(b)$. Let $B=\{u,v,w\}$, where $v < w$ and $u$ is incomparable with both $v$ and $w$, and let $A=\{a,a'\}$ with $a <a'$. Let $P=A \sqcup B$ be the poset obtained by retaining the orders on $A$ and $B$ and adding the relations $a<u$, $a<v$ and $a'<w$. Then $A$ is an ideal and $B$ is a filter, so condition $(a)$ holds. Moreover, $M(a)=\{u,v\}$, while $M(a')=\{w\}$. Condition $(b)$ holds, whereas condition $(c)$ fails for $u \in M(a)$, since $u \not\leq_B w$. 
    \\The same example shows that condition $(c)$ is needed not only to recover the bijection required in condition $(ii)$ of Section \ref{Ladkani's construction}, but also to ensure that the candidate negative relation defines a partial order: without $(c)$, transitivity can fail. Indeed, define a relation $\preceq$ on $A \sqcup B$ by retaining the orders on $A$ and $B$, with the only mixed relations given by \[b \preceq x \iff \text{ there exists } m \in M(x) \text{ such that } b \leq_B m,\] for $x \in A$ and $b \in B$. Then 
    \[u \preceq a \preceq a',\] while $u \npreceq a'$, since $M(a')=\{w\}$ and $u \nleq_Bw$. Thus $\preceq$ is not transitive and does not define a partial order.
\end{remark}
\\[-0.1 cm]
We conclude this subsection with another direct application of Theorem \ref{generalizzazione teomio}. For a finite poset $P$, let $P^{\operatorname{op}}$ denote the poset that has the same underlying set as $P$, with
\[x \leq_{P^{\operatorname{op}}}y \iff y \leq_Px.\]
Recall that a finite poset $P$ has height at most one when it contains no chain $x<y<z$ of three distinct elements. Equivalently, every element of $P$ is minimal or maximal.
A subset $C$ of a poset is called an antichain if no two distinct elements of $C$ are comparable.

For diagrams of finite-dimensional vector spaces over a field, the corresponding equivalence between bounded derived categories already follows from Ladkani \cite[Corollary 4.18]{ladkani3}. The following corollary gives the universal version directly from Theorem \ref{generalizzazione teomio}.

\begin{corollary} 
\label{Two-level posets and opposites}
    Let $P$ be a finite poset of height at most one. Then $P$ and $P^{\operatorname{op}}$ are universally derived equivalent.

    More precisely, if $A=\operatorname{Min}(P)$, $B=P \setminus A$, then $P=A \sqcup B$ is an admissible cut and the associated negative poset is $P^-=P^{\operatorname{op}}$. 
\end{corollary}

\begin{proof}
    Since $A=\operatorname{Min}(P)$, no two distinct elements of $A$ are comparable, and hence $A$ is an antichain. Moreover, if $x \leq_Pa$ for some $a \in A$, then the minimality of $a$ implies $x=a$. Therefore $A$ is an ideal of $P$.

    We first observe that every element of $B$ is maximal in $P$. Indeed, let $b \in B$. Since $b$ is not minimal, there exists $x \in P$ such that $x<_Pb$. If $b$ were not maximal, there would also exist $y \in P$ such that $b <_Py$, giving a chain 
    \[x<_Pb<_Py\] of three distinct elements, contradicting the assumption that $P$ has height at most one. Thus every element of $B$ is maximal. It follows that $B$ is an antichain. Moreover, since every element of $B$ is maximal, $B$ is a filter of $P$. Indeed, if $b \in B$ and $b \leq_Py$, then the maximality of $b$ implies $y=b\in B$.

    Let $a \in A$. Since $B$ is an antichain, every element of $U_B(a)=\{b \in B \colon a \leq_Pb\}$ is minimal in $U_B(a)$. Therefore $M(a)=U_B(a)$. We verify the separation condition. If $m \neq m'$ belong to $M(a)$, then, since $B$ is an antichain, 
    \[[m, \bullet]_B =[\bullet, m]_B=\{m\},\] and similarly
    \[[m', \bullet]_B=[\bullet,m']_B=\{m'\}.\] Hence \[[m, \bullet]_B \cap [m', \bullet]_B= \emptyset\] and \[[\bullet, m]_B \cap [\bullet, m']_B=\emptyset.\] The compatibility condition is also automatic. Indeed, since $A$ is an antichain, if $a \leq_Aa'$, then $a=a'$. Thus, for every $m \in M(a)$, we may take $m'=m$.

    Therefore $P=A \sqcup B$ is an admissible cut.

    It remains to identify the associated negative poset. Let $a \in A$ and $b \in B$. By definition of the negative order,
    \[b \leq_-a \iff \text{ there exists } m \in M(a) \text{ such that } b \leq_Bm.\]
    Since $B$ is an antichain, the relation $b \leq_Bm$ is equivalent to $b=m$. Consequently, 
    \[b \leq_-a \iff b \in M(a).\]
    Since $M(a)=U_B(a)$, this is equivalent to $a \leq_Pb.$
    Hence \[b \leq_-a \iff a \leq_Pb.\]
    The induced orders on $A$ and $B$ are trivial, since both subsets are antichains. 
    Thus every nontrivial relation of $P$ is reversed in $P^-$, and therefore $P^-=P^{\operatorname{op}}$. 
     The universal derived equivalence between $P$ and $P^{\operatorname{op}}$ now follows from Corollary \ref{The associated universally derived equivalent poset}.
\end{proof}

\begin{example} \label{example two level posets}
Let $P$ be the poset of height at most one displayed on the left below; its opposite $P^{\operatorname{op}}$ is displayed on the right.
\[\begin{tikzcd}[sep=small]
	&&&& {\bullet b_1} &&&&&&&&& {\bullet b_1} && \\
	\\
	{\bullet b_3} && {\bullet a_1} &&&& {\bullet a_2} &&& {\bullet b_3} && {\bullet a_1} &&&& {\bullet a_2} \\
	\\
	&&&& {\bullet b_2} &&&&&&&&& {\bullet b_2}
	\arrow[from=1-14, to=3-12]
	\arrow[from=1-14, to=3-16]
	\arrow[from=3-3, to=1-5]
	\arrow[from=3-3, to=3-1]
	\arrow[from=3-3, to=5-5]
	\arrow[from=3-7, to=1-5]
	\arrow[from=3-7, to=5-5]
	\arrow[from=3-10, to=3-12]
	\arrow[from=5-14, to=3-12]
	\arrow[from=5-14, to=3-16]
\end{tikzcd}\]
The poset $P$ has two minimal elements $a_1,a_2$, and three maximal elements $b_1,b_2,b_3$. Notice that its underlying undirected Hasse graph contains a cycle. 

Take \[A=\{a_1,a_2\}, \,\,\,\, B=\{b_1,b_2,b_3\},\] then 
\[M(a_1)=\{b_1,b_2,b_3\}, \,\,\,\, M(a_2)=\{b_1,b_2\}.\] Thus this example uses the non-singleton case of the construction. 
By Corollary \ref{Two-level posets and opposites}, the associated negative poset is $P^{\operatorname{op}}$, obtained by reversing all the covering relations displayed above. Hence $P$ and $P^{\operatorname{op}}$ are universally derived equivalent. They are not isomorphic, since $P$ has two minimal and three maximal elements, whereas $P^{\operatorname{op}}$ has three minimal and two maximal elements.
Thus this example yields a pair of non-isomorphic universally derived equivalent posets whose common underlying undirected Hasse graph contains a cycle.
\end{example}

\subsection{The singleton case and cardinality constraints}
  
\begin{corollary} 
\label{Reconstructing a poset from a cut}
Let $P$ be a finite poset, and let $P=A \sqcup B$ be a partition of the underlying set of $P$.
The following are equivalent: \begin{enumerate}
    \item There exists an order-preserving map $f \colon A \to B$ such that $P=(A \sqcup B, \leq_f^+)$;
    \item $A$ is an ideal of $P$, $B$ is a filter of $P$, and for every $a \in A$, the set $U_B(a)$ has a minimum.
\end{enumerate}

\end{corollary} 
\begin{proof}
Assume $(ii)$. Then $M(a)$ is a singleton for every $a \in A$, say $M(a)=\{f(a)\}$. The separation condition is automatic. If $a \leq_Aa'$, then $f(a') \in U_B(a)$. Since $f(a)$ is the minimum of $U_B(a)$, we obtain $f(a) \leq_Bf(a')$. Thus $f$ is order-preserving and the compatibility condition holds. Theorem \ref{generalizzazione teomio} therefore gives $P=(A \sqcup B, \leq_f^+)$.

Conversely, assume $(i)$. Then, by the definition of $\leq_f^+$, no element of $B$ lies below an element of $A$, so $A$ is an ideal and $B$ is a filter. Moreover, for every $a \in A$, 
\[U_B(a)=\{b \in B \colon f(a) \leq_Bb\}=U_B(f(a)).\]
Hence $f(a)$ is the minimum of $U_B(a)$.

\end{proof}

We say that an admissible cut $P=A \sqcup B$ is of singleton type if $M(a)$ is a singleton for every $a \in A$.
\\[0.8 cm]
\begin{corollary} \label{cornuovo}
    Let $P=A \sqcup B$ be an admissible cut. Suppose that $P$ has a greatest element $\hat{1}$ and that $B \neq \emptyset$. Then  $M(a)$ is a singleton for every $a \in A$. Consequently, the admissible cut is of singleton type.
\end{corollary}

\begin{proof}
Since $B \neq \emptyset$ and $B$ is a filter of $P$, the greatest element $\hat{1}$ belongs to $B$. Fix $a \in A$. Then $a \leq_P \hat{1}$, and hence $\hat{1}\in U_B(a)$. Thus $U_B(a) \neq \emptyset$, so $M(a) \neq \emptyset$. 

    Suppose that $m, m' \in M(a)$. Since $ \hat{1}$ is the greatest element of $P$, we have \[m \leq_B \hat{1} \,\,\, \text{ and }\,\,\, m' \leq_B \hat{1}.\] Therefore \[\hat{1} \in [m, \bullet]_B \cap [m', \bullet]_B.\] By the separation condition of an admissible cut, this is impossible if $m \neq m'$. Hence $m=m'$, and therefore $M(a)$ is a singleton. 

    Since this holds for every $a \in A$, the admissible cut is of singleton type.
\end{proof}

\begin{remark}
    The same argument shows that $|M(a)| \leq 1$ whenever every two elements of $B$ have a common upper bound. Dually, the same conclusion holds when every two elements of $B$ have a common lower bound. In particular, this applies when $B$ is a join-semilattice or a meet-semilattice\footnote{Recall that, for two elements $x,y$ in a poset, their join $x\lor y$, when it exists, is their least upper bound, while their meet $x \land y$, when it exists, is their greatest lower bound. A poset in which every pair of elements admits a join (respectively, a meet) is called a join-semilattice (respectively, a meet-semilattice). For more details, see Garrett Birkhoff \cite[Chapter 1, page 9]{Birkhoff}.}.
\end{remark}

\begin{proposition}
\label{Constant cardinality on connected components}
Let $P=A \sqcup B$ be an admissible cut. If $a, a' \in A$ belong to the same connected component of the underlying undirected Hasse graph of $A$, then \[|M(a)|=|M(a')|.\]
In particular, on each connected component $C$ of $A$, there exists an integer $r_C \geq0$ such that 
\[|M(a)|=r_C \,\,\, \text{ for every } a \in C.\]
Consequently, if $A$ is connected and $M(a_0)$ is a singleton for some $a_0 \in A$, then $M(a)$ is a singleton for every $a \in A$, and the admissible cut is of singleton type.
\end{proposition}
\begin{proof}
By Theorem \ref{generalizzazione teomio}, the sets $M(a)$ and $M(a')$ are in bijection whenever $a,a' \in A$ are comparable. Since adjacent vertices in the Hasse diagram of $A$ are comparable, the function $a \mapsto|M(a)|$ is constant on each connected component of $A$. The remaining assertions follow immediately.  
\end{proof}

\begin{remark}
    The connectedness assumption in Proposition \ref{Constant cardinality on connected components} cannot be omitted. Indeed, in Example \ref{example two level posets} the poset $A=\{a_1,a_2\}$ has two connected components, while $|M(a_1)|=3$ and $|M(a_2)|=2$. Thus the cardinality of $M(a)$ need not be constant on different connected components of $A$.
\end{remark}

\subsection{Invariance under isomorphisms}

The finite procedure described after Theorem \ref{generalizzazione teomio} may produce isomorphic outputs from different choices of data. The following result identifies a basic source of this redundancy: replacing the two input posets by isomorphic copies and transporting the construction data along the chosen isomorphisms does not change the isomorphism classes of the resulting posets. In particular, when the input posets are fixed, families lying in the same orbit under the natural action of \[\operatorname{Aut}(A) \times \operatorname{Aut}(B)\] produce isomorphic positive and negative posets. Thus, such families need not be considered separately when classifying the outcomes of the construction. 

Let $A$ and $B$ be finite posets, and let $\mathcal{Y}=(Y_a)_{a\in A}$ be a family of subsets of $B$ satisfying the hypotheses of Section \ref{Ladkani's construction}. We denote the associated positive and negative posets by \[(A \sqcup B)^+_{\mathcal{Y}} \,\,\, \text{ and }\,\,\, (A \sqcup B)^-_{\mathcal{Y}}.\] 
When it is necessary to distinguish the corresponding orders, we write $\leq^+_{\mathcal{Y}}$ and $\leq^-_{\mathcal{Y}}$.

\begin{proposition} \label{invgen}
    Let $A, B, A'$ and $B'$ be finite posets, and let $\mathcal{Y}=(Y_a)_{a \in A}$ be a family of subsets of $B$ satisfying the hypotheses of the construction in Section \ref{Ladkani's construction}. Let
    \[\alpha \colon A \to A', \,\,\,\,\,\,\, \beta \colon  B \to B'\] be isomorphisms of posets. 
    For every $a' \in A'$, define $Y'_{a'} \coloneq \beta (Y_{\alpha^{-1}(a')})$, and set $\mathcal{Y'}=(Y'_{a'})_{a' \in A'}$. 
    Then $\mathcal{Y}'$ satisfies the same hypotheses. Moreover,
    \[(A \sqcup B ) _{\mathcal{Y}}^+\cong (A' \sqcup B' )_{\mathcal{Y}'}^+\] and \[(A \sqcup B) _{\mathcal{Y}}^- \cong (A' \sqcup B')_{\mathcal{Y}'}^-.\]
\end{proposition}
\begin{proof}
    We first prove that $\mathcal{Y}'$ satisfies the same hypotheses.
    
    Let $a' \in A'$, and let $y'_1,y'_2 \in \mathcal{Y}'_{a'}$ be distinct. Set $a=\alpha^{-1}(a')$. By the definition of $Y'_{a'}$, there exist distinct elements $y_1,y_2 \in Y_a$ such that \[y'_1=\beta(y_1), \,\,\,\,\,\, y_2'=\beta(y_2).\]
    Since $\beta \colon B \to B'$ is an isomorphism of posets, it maps upper and lower cones in $B$ onto the corresponding cones in $B'$: 
    \[[\beta(y), \bullet]_{B'}=\beta([y,\bullet]_B), \,\,\,\,\,\, [\bullet, \beta(y)]_{B'}=\beta([\bullet,y]_B).\]
    Since $\mathcal{Y}$ satisfies condition (i),
    \[[y_1,\bullet]_B \cap [y_2,\bullet]_B= \emptyset \,\,\, \text{ and } \,\,\, [\bullet, y_1]_B \cap [\bullet, y_2]_B= \emptyset.\]
    Applying $\beta$, we obtain 
    \[[y'_1,\bullet]_{B'} \cap [y'_2,\bullet]_{B'}= \emptyset \,\,\, \text{ and } \,\,\, [\bullet, y'_1]_{B'} \cap [\bullet, y'_2]_{B'}= \emptyset.\]
    Thus condition (i) holds for $\mathcal{Y}'.$
    We now verify condition (ii). Let $a',a'' \in A'$ with $a' \leq_{A'} a''$. Since $\alpha \colon A \to A'$ is an isomorphism,
    \[\alpha^{-1}(a') \leq_A \alpha^{-1}(a'').\] Hence there exists a bijection 
    \[\varphi_{\alpha^{-1}(a'),\alpha^{-1}(a'')} \colon Y_{\alpha^{-1}(a')} \to Y_{\alpha^{-1}(a'')}\] such that 
    \[y \leq_B \varphi_{\alpha^{-1}(a'),\alpha^{-1}(a'')}(y)\] for every $y \in Y_{\alpha^{-1}(a')}$. Define
    \[\varphi'_{a',a''} \coloneq \beta \circ \varphi_{\alpha^{-1}(a'), \alpha^{-1}(a'')} \circ \beta^{-1}.\]
    Then $\varphi_{a',a''}' \colon Y'_{a'} \to Y'_{a''}$ is a bijection.
    
    If $z' \in Y_{a'}'$, then $z'=\beta(z)$ for some $z \in Y_{\alpha^{-1}(a')}$. Therefore,
    \[z \leq_B\varphi_{\alpha^{-1}(a'),\alpha^{-1}(a'')}(z).\] Since $\beta$ is order-preserving, $z' \leq_{B'} \varphi'_{a',a''}(z')$.
    Thus condition (ii) also holds for $\mathcal{Y}'$.

    It remains to prove the two isomorphisms. Define \[\gamma \colon A \sqcup B \to A' \sqcup B'\] by \[\gamma(a)=\alpha(a) \,\, \text{ for } a \in A, \,\,\,\,\, \gamma(b)=\beta(b) \,\, \text{ for } b \in B.\]
    Since $\alpha$ and $\beta$ are isomorphisms, $\gamma$ is a bijection and preserves the orders inside $A$ and $B$. 

    We first consider the positive mixed relations. Let $a \in A$ and $b \in B$. Then $a \leq^+_{\mathcal{Y}}b$ if and only if there exists $y \in Y_a$ such that $y \leq_B b$.

    Since $\beta \colon B \to B'$ is an isomorphism, this is equivalent to the existence of $y' \in \beta( Y_a)$ such that $y' \leq_{B'} \beta(b)$.

    Moreover, $\beta(Y_a)=Y'_{\alpha(a)}$, because \[Y'_{\alpha(a)}=\beta(Y_{\alpha^{-1}(\alpha(a))})=\beta(Y_a).\]
    Hence \[a \leq_{\mathcal{Y}}^+b \iff \alpha(a) \leq^+_{\mathcal{Y}'} \beta(b),\]

    or equivalently \[a \leq^+_{\mathcal{Y}}b \iff \gamma(a) \leq^+_{\mathcal{Y}'} \gamma (b).\]
    Therefore $\gamma$ is an isomorphism 
    \[(A \sqcup B)^+_{\mathcal{Y}} \cong (A' \sqcup B')^+_{\mathcal{Y}'}.\] 
    Similarly, for the negative mixed relations, \[b \leq_{\mathcal{Y}}^-a \iff \text{ there exists } y \in Y_a \text{ such that } b \leq_B y \iff\] \[ \text{ there exists } y' \in Y'_{\alpha(a)} \text{ such that } \beta (b) \leq_{B'} y' \iff \beta (b) \leq^-_{\mathcal{
    Y'}} \alpha(a).\] 
    Thus \[b \leq^-_{\mathcal{
    Y}}a \iff \gamma (b) \leq ^- _{\mathcal{
    Y}'} \gamma (a),\] and consequently \[(A \sqcup B)^-_{\mathcal{Y}} \cong (A' \sqcup B')^-_{\mathcal{Y}'}.\]
\end{proof}

The singleton case gives the following immediate consequence.
\begin{corollary}
    Let $A$ and $B$ be finite posets, let $f \colon A \to B$ be an order-preserving map, and let $\alpha \in \operatorname{Aut}(A)$, $\beta \in \operatorname{Aut}(B)$. 

    Define $g=\beta \circ f \circ \alpha^{-1}$. Then
    \[(A \sqcup B, \leq_f^+)\cong (A \sqcup B, \leq_g^+)\] and 
    \[(A \sqcup B, \leq_f^-) \cong (A \sqcup B, \leq_g^-).\]
\end{corollary}
\begin{proof}
    Apply Proposition \ref{invgen} with $A'=A, B'=B$, and with the singleton family $Y_a=\{f(a)\}$. Then
    \[Y'_a=\beta(Y_{\alpha^{-1}(a)})=\beta(\{f(\alpha^{-1}(a))\})=\{g(a)\}.\] The transformed family is the singleton family associated with $g$, and the isomorphisms follow from Proposition \ref{invgen}.
\end{proof}

Thus, in the singleton case, order-preserving maps which lie in the same orbit under the natural action of $\operatorname{Aut} (A)\times \operatorname{Aut}(B)$ give rise to isomorphic positive and negative posets. More generally, Proposition \ref{invgen} shows that the same invariance holds for arbitrary families satisfying the hypotheses of Section \ref{Ladkani's construction}.

\begin{remark}
    Proposition \ref{invgen} includes the case of an automorphism of the whole poset. Indeed, let $P=(A \sqcup B)_{\mathcal{Y}}^+$ and let $\sigma \in \operatorname{Aut}(P)$. Set $A'=\sigma (A)$, $B'=\sigma(B)$, and take $\alpha=\sigma \mid_A \colon A \to A'$ and $\beta=\sigma \mid_B \colon B \to B'$. Then the transformed family satisfies $Y'_{\sigma(a)}=\sigma(Y_a)$ for every $a \in A$, and Proposition \ref{invgen} gives the corresponding positive and negative isomorphisms.
\end{remark}

%% file: orientation_of_tree.tex
\section{Source-clicks and tree orientations}

In this section, we study source-to-sink transformations from the point of view of admissible cuts. We first consider an arbitrary finite poset $P$ and a minimal element $s$, and give a criterion for reversing the covering relations incident with $s$ to be realized by the construction of Theorem \ref{generalizzazione teomio}. Whenever this criterion is satisfied, the poset obtained by the source-click is universally derived equivalent to $P$.

Source-to-sink transformations are classical in the representation theory of quivers, where they arise in the Bernstein--Gelfand--Ponomarev reflection functors \cite{BGP}. Their relation with derived equivalences was further developed by Happel \cite{Happel}. 

We then specialize to orientations of finite trees. In this case, the criterion is automatically satisfied at every source, so each source-click is realized by an admissible cut. We use the classical fact that any orientation of a finite tree can be reached from any other by a finite sequence of such clicks, and include a self-contained proof for completeness. This recovers Ladkani's result \cite[Corollary 1.8]{ladkani} that all orientations of a finite tree are universally derived equivalent. Related orientation-independence results have also been obtained in the stronger framework of strong stable equivalence; see Groth and Šťovíček \cite[Corollary 9.23]{Groth}.

Finally, we discuss why the analogous transitivity statement fails in general for graphs with cycles.

\subsection{An intrinsic criterion for source-clicks} \label{An intrinsic criterion for source-clicks}

Let $P$ be a finite poset and let $s$ be a minimal element of $P$. We denote by $\operatorname{Cov}_P(s)$ the set of elements covering $s$, that is, \[\operatorname{Cov}_P(s) \coloneq \{m \in P \colon s <_Pm \text{ and there is no } x \in P \text{ with } s<_Px<_Pm\}.\]

Since $s$ is a source in the oriented Hasse diagram of $P$, reversing all the covering relations incident with $s$ does not create an oriented cycle: after the reversal, $s$ becomes a sink. We denote by $\operatorname{cl}_s(P)$ the poset obtained by reversing every covering relation \[s <_Pm, \,\,\,\, m \in \operatorname{{Cov}_P(s)},\]
and taking the transitive closure. We call this operation a source-click at $s$.

The source-to-sink equivalence is the case $X=\{*\}$ of Ladkani's construction \cite[Corollary 1.7]{ladkani}. The point of the following proposition is to formulate this construction intrinsically for a given poset $P$: for a minimal element $s$, the relevant subset is necessarily $\operatorname{Cov}_P(s)$, and the admissible-cut condition gives a recognition criterion entirely in terms of the order structure of $P$.

\begin{proposition}
\label{General source-click criterion}
Let $P$ be a finite poset, let $s$ be a minimal element of $P$, and set
\[A=\{s\}, \,\,\,\, B=P \setminus\{s\}.\] Then \[M(s)=\operatorname{Cov}_P(s).\]
Moreover, the following conditions are equivalent:
\begin{enumerate}
    \item the partition $P=A \sqcup B$ is an admissible cut;
    \item for every pair of distinct elements $m,m' \in \operatorname{Cov_P(s)},$
    \[[m,\bullet]_B \cap [m', \bullet]_B= \emptyset\] and \[[\bullet, m]_B \cap [\bullet, m']_B = \emptyset.\]
\end{enumerate}
    Whenever these equivalent conditions hold, the associated negative poset coincides with the source-click of $P$ at $s$:
    \[P^-=\operatorname{cl}_s(P).\]
    In particular, $P$ and $\operatorname{cl}_s(P)$ are universally derived equivalent.
    
\end{proposition}

\begin{proof}
   Since $s$ is minimal, $A=\{s\}$ is an ideal of $P$, and therefore $B=P \setminus\{s\}$ is a filter.

   We first identify $M(s)$. By definition, \[M(s)=\operatorname{Min}_B\{b \in B \colon s \leq_Pb\}.\] An element $b \in B$ is minimal among the elements above $s$ if and only if there is no element $x \in P$ such that $s<_Px<_Pb$. This is equivalent to saying that $b$ covers $s$. Hence \[M(s)= \operatorname{Cov}_P(s).\]
   The ideal-filter condition is therefore automatic. Moreover, the compatibility condition is vacuous, since $A$ consists of a single element. Consequently, the partition $P=A \sqcup B$ is admissible if and only if the separation condition holds for the distinct elements of \[M(s)=\operatorname{Cov}_P(s),\] which proves the equivalence between $(i)$ and $(ii)$. 

   Assume now that these equivalent conditions hold. By Corollary \ref{The associated universally derived equivalent poset}, the mixed relations in the associated negative poset are
   \[b\leq_-s \iff \text{ there exists } m \in \operatorname{Cov}_P(s) \text{ such that } b \leq_B m.\]
   On the other hand, in the source-click $\operatorname{cl}_s(P)$, every edge $s<_Pm$ is replaced by $m<s$. Thus, for $b \in B$, there is an increasing path from $b$ to $s$ in $\operatorname{cl}_s(P)$ if and only if there exists $m \in \operatorname{Cov}_P(s)$ such that $b \leq_Bm$. Therefore 
   \[b \leq_{\operatorname{cl}_s(P)}s \iff b \leq_-s.\] The orders induced on $B$ are unchanged in both posets. Indeed, since $B$ is a filter of $P$, every chain in $P$ between elements of $B$ lies entirely in $B$, while after the source-click $s$ is a sink and therefore cannot create new relations between elements of $B$. Moreover, neither poset has a mixed relation of the form $s \leq b$. Hence $P^-= \operatorname{cl}_s(P)$. The universal derived equivalence follows from Corollary \ref{The associated universally derived equivalent poset}.

\end{proof}

\subsection{Tree orientations}

\textbf{Setup and notation.} Let $T$ be a finite tree with vertex set $V$ and edge set $E$. An orientation of $T$ is a partial order $P$ on $V$ whose covering relation, viewed as an undirected graph, equals $T$. Equivalently, an orientation is obtained by directing each edge of $T$ and taking the transitive closure. Since $T$ is acyclic, this always gives a partial order. We denote by $\operatorname{Or}(T)$ the set of orientations of $T$. A vertex $s$ is called a source of $P$ if it is minimal in $P$, i.e., for every edge joining $s$ to a vertex $v$, it holds that $s <_Pv$. Dually, $s$ is a sink if it is maximal in $P$. 

    Notice that, given $P \in \operatorname{Or(T)}$ and $x \neq y \in V$, $x \leq_Py$ if and only if the order increases at every step along the unique path from $x$ to $y$. 

For $P \in \operatorname{Or(T)}$ and a source $s$ of $P$, the source-click $cl_s(P)$ defined in Section \ref{An intrinsic criterion for source-clicks} amounts to reversing all edges incident to $s$; we also refer to this operation as a click at $s$.

\begin{corollary}
\label{Source step}
Let $T$ be a finite tree, let $P \in \operatorname{Or(T)}$ with $|V| \geq 2$ and let $s$ be a source of $P$. Set $A=\{s\}$ and $B=V \setminus \{s\}$. Then $P=A \sqcup B$ is an admissible cut. Moreover, $M(s)=N(s)$, where $N(s)$ is the set of neighbours of $s$ in $T$, and the associated negative poset is the source-click of $P$ at $s$: \[P^-=\operatorname{cl}_s(P).\] In particular, $P$ and $cl_s(P)$ are universally derived equivalent.   
\end{corollary}

\begin{proof}
    Since the covering relation of $P$, viewed as an undirected graph, is $T$, and $s$ is a source, the elements covering $s$ are exactly its neighbours in $T$. Hence
    \[\operatorname{Cov}_P(s)=N(s).\] 
    Let $m \neq m'$ be two neighbours of $s$. The vertices $m$ and $m'$ belong to distinct connected components of $T \setminus \{s\}$. 

    Any element of $B$ which is comparable with $m$ through an increasing path in the Hasse diagram belongs to the same component of $T \setminus \{s\}$ as $m$. The same holds for $m'$. Since the two components are disjoint, both the upper and the lower cones of $m$
 and $m'$ in $B$ are disjoint:
\[[m, \bullet]_B \cap [m', \bullet]_B= \emptyset,\]
\[[\bullet, m]_B \cap [\bullet,m']_B= \emptyset.\] The conclusion now follows from Proposition \ref{General source-click criterion}.
\end{proof}
For $P,Q \in \operatorname{Or}(T)$, we write $P \to Q$ if $Q$ is obtained from $P$ by clicking at a source; that is, if $Q=cl_s(P)$ for some source $s$ of $P$. We denote by $\to ^*$ the reflexive-transitive closure of this relation. Thus $P \to ^*Q$ means that there exists a finite sequence of orientations \[P=P_0 \to P_1 \to \dots \to P_r=Q,\] or equivalently that $Q$ can be reached from $P$ by finitely many source-clicks.
The following combinatorial fact is classical; see \cite{BGP}. We include a self-contained proof for completeness.

\begin{proposition}
\label{Transitivity of clicks on a tree}
    Let $T$ be a finite tree. For every $P,Q \in \operatorname{Or}(T)$, one has $P \to^*Q$.
\end{proposition}
\begin{proof}
    We proceed by induction on $n=|V(T)|$. The cases $n=1$ and $n=2$ are immediate. Assume $n \geq 3$. Fix a leaf $\ell$ of $T$ that has a unique neighbour $p$ and set $T'=T\setminus\ell$, which is a tree with $n-1$ vertices.

    Every orientation $P \in \operatorname{Or}(T)$ is determined by its restriction $P' \in \operatorname{Or}(T')$, together with the direction of the edge $\ell p$. We say that the leaf is in the state $L$ if $\ell <_Pp$, and in state $H$ if $p<_P\ell$. We write $P=(P',L)$ or $P=(P',H)$.

    Interactions of clicks with this decomposition:
    \begin{itemize}
        \item $\ell$ is in state $L$ if and only if $\ell$ is a source of $P$. In this case, clicking $\ell$ changes the state from $L$ to $H$ without changing the restriction to $T'$.
        \item A vertex $v\in V(T') \setminus \{p\}$ is a source of $P$ if and only if it is a source of $P'$. (Because its edges coincide in $T$ and $T'$.) In this case, clicking $v$ induces on $T'$ the same operation as clicking $v$ in $P'$, and it leaves the leaf state unchanged.
        \item The vertex $p$ is a source of $P$ if and only if it is a source of $P'$ and the leaf is in the state $H$. In this case, clicking $p$ induces the click at $p$ on $T'$ and changes the leaf state from $H$ to $L$.
    \end{itemize}

By the induction hypothesis, any two orientations of $T'$ can be connected by a sequence of source-clicks. Let $\pi$ be such a sequence starting from $P'$. We lift $\pi$ to $T$ by performing its clicks in the same order and inserting, when necessary, clicks at the leaf $\ell$, which leave the restriction to $T'$ unchanged. We read $\pi$ from left to right. 

A click at vertex $v \neq p$ is lifted to the same click at $v$ in $T$. 

A click at $p$ is lifted as follows: if the current leaf state is $H$, click $p$; if the current leaf state is $L$, first click $\ell$, changing the state to $H$, and then click $p$. In this way the restriction to $T'$ follows exactly the sequence $\pi$.

Moreover, after every lifted click at $p$, the leaf state is $L$. Clicks at vertices different from $p$ and $\ell$ do not change the leaf state. Thus, after lifting $\pi$, the final restriction to $T'$ is the desired one, and the final state is determined explicitly.

Now let $Q=(Q',\eta) \in \operatorname{Or}(T)$ be arbitrary, with state $\eta \in \{L,H\}$. By induction, choose a source-click sequence in $T'$ from $P'$ to $Q'$, and lift it to $T$ as above. This gives an orientation of the form $(Q', \epsilon)$. If $\epsilon = \eta$, we are done. If $\epsilon=L$ and $\eta=H$, click the leaf $\ell$, which changes only the leaf state.

It remains only to see that the opposite change, from $H$ to $L$ at fixed $Q'$, is also possible. Choose an orientation $R' \in \operatorname{Or}(T')$ in which $p$ is a source, for example by orienting every edge away from $p$. By the induction hypothesis, there is a source-click sequence from $Q'$ to $R'$, and another one from $cl_p(R')$ back to $Q'$. Concatenating these sequences with the click at $p$ gives a closed source-click walk based at $Q'$, that is, a sequence of source-clicks on $T'$ that starts and ends at $Q'$. This walk contains a click at $p$. Lifting this closed walk from the state $(Q',H)$ returns the restriction to $Q'$, and after the last lifted click at $p$, the leaf state is $L$. Hence $(Q',H) \to^*(Q',L)$. 

Therefore both leaf states are reachable over the same restriction $Q'$, and so $P \to ^*Q$. This completes the induction.
\end{proof}
\begin{theorem} 
\label{tree-exhaustiveness}
    Let $T$ be a finite tree. Starting from any orientation $P \in \operatorname{Or}(T)$, every orientation of $T$ can be reached by a finite sequence of applications of the general construction of Theorem \ref{generalizzazione teomio} to single-source decompositions. In particular, all orientations of $T$ are universally derived equivalent.
\end{theorem}
\begin{proof}
    Let $P,Q \in \operatorname{Or}(T).$ By Proposition \ref{Transitivity of clicks on a tree}, there exists a finite sequence of source-clicks from $P$ to $Q$. By Corollary \ref{Source step}, each source-click is realized by the construction of Theorem \ref{generalizzazione teomio} applied to a decomposition of the form $A=\{s\}$, where $s$ is a source of the current orientation. Hence consecutive orientations in the sequence are universally derived equivalent. Since universal derived equivalence is transitive, $P$ and $Q$ are universally derived equivalent. Therefore all orientations of $T$ are universally derived equivalent.
\end{proof}

\begin{remark}
    If $s$ is a source with at least two neighbours, then $M(s)=N(s)$ has more than one element. Thus the source step of Corollary \ref{Source step} cannot, in general, be obtained from the singleton case of Corollary \ref{Reconstructing a poset from a cut}. Already in the path $A_3$, the middle vertex can be the unique source, and clicking it requires the general construction.
\end{remark}

\begin{remark}
   The transitivity argument above uses that the underlying graph is a tree. For a graph with cycles, the analogous statement for acyclic orientations and graph-theoretic source-clicks can fail.

   Indeed, suppose that a graph contains a cycle $C$, and choose one of the two possible directions around $C$. For an acyclic orientation $O$, define $\sigma_C(O)$ as the number of edges of $C$ oriented in the chosen direction minus the number of edges oriented in the opposite direction. 

   A graph-theoretic source-click does not change $\sigma_C$. If the clicked vertex does not lie on $C$, this is clear. If it lies on $C$, then the two edges of $C$ incident with it are both reversed. Since the clicked vertex is a source, before the click one of these two edges is oriented in the chosen direction and the other one is oriented in the opposite direction. After the click, the two contributions are exchanged. Hence $\sigma_C$ is invariant under graph-theoretic source-clicks.

   On the other hand, acyclic orientations with different values of $\sigma_C$ exist. Indeed, write $C=(v_1, \dots, v_r,v_1)$ in the chosen direction and orient every edge according to a total order extending $v_1<\dots<v_r$. This gives $\sigma_C=r-2$, while reversing the total order gives $\sigma_C=2-r$.

   Therefore, for graphs with cycles, such clicks cannot in general connect all acyclic orientations. Notice that, unlike the tree case, a graph-theoretic source-click need not coincide with the poset source-click $\operatorname{cl_s}$ of Section \ref{An intrinsic criterion for source-clicks}: an edge of an acyclic orientation may become redundant after taking the transitive closure and hence not correspond to a covering relation. Thus $\sigma_C$ is an obstruction to transitivity for graph orientations, rather than in general an invariant of the poset source-click operation.

   In particular, the presence of a cycle does not by itself prevent an individual poset source-click from being realized by an admissible cut; this is governed by the criterion of Proposition \ref{General source-click criterion}.
\end{remark}

%% file: TDA.tex
\section{Application to persistence modules and parameter extension} \label{Application to persistence modules and parameter extension}
Persistence modules provide a natural connection between the diagram categories considered above and topological data analysis. Let $k$ be a field. We denote by $Vect_k$ the category of $k$-vector spaces, and by $vec_k$ its full subcategory of finite-dimensional $k$-vector spaces. A persistence module indexed by a poset $P$ is a functor $P \longrightarrow Vect_k$, and it is called pointwise finite-dimensional when all its values are finite-dimensional \cite[Section 2.3]{Botnan}. Thus, for a finite poset $P$, the diagram category $vec_k^P$ is the category of pointwise finite-dimensional persistence modules indexed by $P$.

Since $vec_k$ is an abelian category, if two finite posets $P$  and $Q$ are universally derived equivalent, then, for every field $k$, \[D(vec_k^P) \simeq D(vec_k^Q). \] Hence every universal derived equivalence between finite posets obtained above yields, in particular, a derived equivalence between the corresponding categories of pointwise finite-dimensional persistence modules.

For $n \geq 1$, let $I_n$ denote the $n$-element chain. 
When $P$ is a product of finite chains, this recovers the usual discrete multiparameter setting: a product $I_{n_1} \times \dots \times I_{n_d}$ is a finite $d$-dimensional parameter grid \cite[Sections 2.2-2.3]{Botnan}. More generally, arbitrary finite posets allow parameter spaces with branching or nonstandard orientations.


The following result is a special case of Ladkani \cite[Lemma 2.9]{ladkani}. We include a short argument for completeness.

\begin{proposition}
\label{Stability under parameter extension}
    Let $X$ and $Y$ be finite universally derived equivalent posets, and let $R$ be any finite poset. Then
    \[X \times R \,\,\, \text{ and } \,\,\, Y \times R\] are universally derived equivalent.
\end{proposition}

\begin{proof}
    Let $\mathcal{A}$ be an arbitrary abelian category. Since $R$ is finite, the diagram category $\mathcal{A}^R$ is again abelian. Moreover, there are natural isomorphisms of categories 
    \[\mathcal{A}^{X \times R} \cong (\mathcal{A}^R)^X \,\,\, \text{ and }\,\,\, \mathcal{A}^{Y \times R} \cong (\mathcal{A}^R)^Y.\] Since $X$ and $Y$ are universally derived equivalent, we may apply the defining equivalence to the abelian category $\mathcal{A}^R$. Hence \[D((\mathcal{A}^R)^X) \simeq D((\mathcal{A}^R)^Y).\] Using the natural identification above, we obtain \[D(\mathcal{A}^{X \times R}) \simeq D(\mathcal{A}^{Y \times R}).\] Since $\mathcal{A}$ was arbitrary, $X \times R$ and $Y \times R$ are universally derived equivalent. 
\end{proof}
\begin{corollary} 
\label{Parameter extension for admissible cuts}
    Let $P=A \sqcup B$ be an admissible cut and let $P^-$ be its associated negative poset. Then, for every finite poset $R$, 
    \[P \times R \,\,\, \text{ and } \,\,\, P^- \times R\] are universally derived equivalent.
\end{corollary}

\begin{proof}
    By Corollary \ref{The associated universally derived equivalent poset}, $P$ and $P^-$ are universally derived equivalent. The conclusion then follows directly from Proposition \ref{Stability under parameter extension}.
\end{proof}

In the persistence setting, Corollary \ref{Parameter extension for admissible cuts} allows one to adjoin further discrete parameters to any universally derived equivalent pair produced by an admissible cut. In particular, taking
\[R=I_{n_1} \times \dots \times I_{n_d}\] yields derived equivalences between categories of persistence modules obtained by adjoining $d$ ordinary discrete persistence parameters.

\begin{example} \label{example Tamari TDA}
    Consider the admissible cut of the Tamari lattice $T_3$ from Example \ref{T3/D5}, and let $T_3^-$ be its associated negative poset. By Corollary \ref{Parameter extension for admissible cuts}, for every finite poset $R$,
    \[T_3 \times R \,\,\, \text{ and } \,\,\, T_3^- \times R\] are universally derived equivalent.

    In particular, taking $R=I_n$ yields, for every $n \geq 1$, a pair of universally derived equivalent posets with $5n$ elements. Thus the equivalence of Example \ref{T3/D5} gives rise to an infinite family of universally derived equivalent pairs. 

    More generally, if
    \[R=I_{n_1} \times \dots \times I_{n_d},\] then, for every field $k$,
    \[D(vec_k^{T_3 \times R}) \simeq D(vec_k^{T_3^- \times R}).\] Thus, the equivalence of Example \ref{T3/D5} extends to parameter spaces obtained by adjoining any finite number of ordinary discrete parameters.

    When $R$ is a product of nonempty finite chains, these two parameter posets are still non-isomorphic. Indeed, $T_3 \times R$ has a unique maximal element, whereas $T_3^- \times R$ has two maximal elements. Thus the construction yields infinite families of non-isomorphic parameter posets with equivalent derived categories of pointwise finite-dimensional persistence modules.
\end{example}

\begin{corollary}
\label{Parameter extension for tree orientation}
    Let $T$ be a finite tree, let $P,Q \in \operatorname{Or}(T)$, and let $R$ be any finite poset. Then
    \[P \times R \,\,\, \text{ and } Q \times R\] are universally derived equivalent.
\end{corollary}

\begin{proof}
    By Theorem \ref{tree-exhaustiveness}, $P$ and $Q$ are universally derived equivalent. The conclusion follows from Proposition \ref{Stability under parameter extension}.
\end{proof}

In particular, if
\[R=I_{n_1} \times \dots \times I_{n_d},\] then the posets $P \times R$, as $P$ ranges over $\operatorname{Or(T)}$, yield derived equivalent categories of pointwise finite-dimensional persistence modules. Thus, after adjoining any finite number of ordinary discrete parameters, the orientation of the tree coordinate is not detected by the derived category.

The following example illustrates Corollary \ref{Parameter extension for tree orientation} in the simplest nontrivial two-parameter case.
\begin{example}
    Let $I_3$ and $Z_3$ be the following two orientations of the three-vertex tree, where $I_3$ is the three-element chain, and let $I_2$ denote the two-element chain. Set
    \[G \coloneq I_3 \times I_2, \,\,\,\,\, H \coloneq Z_3 \times I_2.\]
   
\[\begin{tikzcd}[sep=small]
	&& {I_3} &&&&&&&& {Z_3} && \\
	&& {x_1} &&&&&&&& {z_1} \\
	{x_0} &&&& {x_2} &&&& {z_0} &&&& {z_2}
	\arrow[from=2-3, to=3-5]
	\arrow[from=3-1, to=2-3]
	\arrow[from=3-9, to=2-11]
	\arrow[from=3-13, to=2-11]
\end{tikzcd}\]
\\

\[\begin{tikzcd}[column sep=small,row sep=scriptsize]
	&& {G=I_3 \times I_2} &&&&&&&& {H=Z_3 \times I_2} && \\
	{x_0} && {x_1} && {x_2} &&&& {z_0} && {z_1} && {z_2} \\
	\\
	{x_0'} && {x_1'} && {x_2'} &&&& {z_0'} && {z_1'} && {z_2'}
	\arrow[from=2-1, to=2-3]
	\arrow[from=2-1, to=4-1]
	\arrow[from=2-3, to=2-5]
	\arrow[from=2-3, to=4-3]
	\arrow[from=2-5, to=4-5]
	\arrow[from=2-9, to=2-11]
	\arrow[from=2-9, to=4-9]
	\arrow[from=2-11, to=4-11]
	\arrow[from=2-13, to=2-11]
	\arrow[from=2-13, to=4-13]
	\arrow[from=4-1, to=4-3]
	\arrow[from=4-3, to=4-5]
	\arrow[from=4-9, to=4-11]
	\arrow[from=4-13, to=4-11]
\end{tikzcd}\]
 By Theorem \ref{tree-exhaustiveness}, the posets $I_3$ and $Z_3$ are universally derived equivalent, since they are two orientations of the same finite tree. Therefore, by Proposition \ref{Stability under parameter extension}, the product posets $G$ and $H$ are universally derived equivalent.

    The poset $G$ is the usual $3 \times 2$ rectangular grid, which is a standard finite parameter space for discrete two-parameter persistence \cite[Sections 2.2-2.3]{Botnan}. By contrast, $H$ combines one ordinary discrete parameter with one zigzag-shaped parameter. Thus $G$ and $H$ provide an example of two non-isomorphic parameter posets whose categories of pointwise finite-dimensional persistence modules have equivalent derived categories. 

    To see that $G$ and $H$ are not isomorphic, it is enough to compare their minimal elements: $G$ has a unique minimal element, whereas $H$ has two minimal elements. This shows that universal derived equivalence may relate a standard multiparameter grid to a parameter poset with a nonstandard orientation.

\end{example}

\section*{Acknowledgments}
I would like to thank Jorge Vitória for his invaluable guidance, many helpful discussions, and numerous suggestions since the early stages of this work. I am also grateful to Eero Hyry for reading an earlier version of the manuscript and for useful comments.